\documentclass{amsart}

\usepackage{graphicx}
\usepackage{amsmath}%
\usepackage{amsfonts}%
\usepackage{amssymb}%
\usepackage{graphicx}
\usepackage{enumerate}
\usepackage{enumitem}

\usepackage{tikz}
\usepackage[]{subfig}
\usepackage{diagbox}
\usepackage[rightcaption]{sidecap}
\usepackage{wrapfig}

\usepackage{pgfplots}
\pgfplotsset{compat=1.9}

\usepackage{mathtools}

\DeclarePairedDelimiter\floor{\lfloor}{\rfloor}

\newtheorem{theorem}{Theorem}[section]
\newtheorem{lemma}[theorem]{Lemma}

\theoremstyle{definition}
\newtheorem{definition}[theorem]{Definition}
\newtheorem{example}[theorem]{Example}

\newtheorem{corollary}[theorem]{Corollary}
\newtheorem{proposition}[theorem]{Proposition}

\newtheorem{conjecture}[theorem]{Conjecture}

\theoremstyle{remark}
\newtheorem{remark}[theorem]{Remark}

\numberwithin{equation}{section}

\begin{document}

\title{Coincidence and Self-coincidence of Many Maps between Digital Images}

\author{Muhammad Sirajo Abdullahi}

\curraddr[]{\textit{KMUTTFixed Point Research Laboratory},  KMUTT-Fixed Point Theory and Applications Research Group, SCL 802 Fixed Point Laboratory, Department of Mathematics, Faculty of Science, King Mongkut's University of Technology Thonburi (KMUTT), 126 Pracha-Uthit Road, Bang Mod, Thrung Khru, Bangkok 10140, Thailand}
\email{abdullahi.sirajo@udusok.edu.ng [M.S. Abdullahi]}
\email{poom.kumam@mail.kmutt.ac.th [P. Kumam]}

\author{Poom Kumam} 
\curraddr{Center of Excellence in \textit{Theoretical and Computational Science (TaCS-CoE)},  Science Laboratory Building, King Mongkut's University of Technology Thonburi (KMUTT), 126 Pracha-Uthit Road, Bang Mod, Thrung Khru, Bangkok 10140, Thailand}
\email{poom.kumam@mail.kmutt.ac.th [P. Kumam]}

\address[]{Department of Mathematics, Faculty of Science, Usmanu Danfodiyo University, Sokoto, Nigeria}
\email{abdullahi.sirajo@udusok.edu.ng [M.S. Abdullahi]}
\email{garba.isa@udusok.edu.ng [I.A. Garba]}



\author{Isah Abor Garba} 


\thanks{The first author was supported by the ``Petchra Pra Jom Klao Ph.D. Research Scholarship from King Mongkut's University of Technology Thonburi".}






\begin{abstract}
The aim of this paper is to generalize some of the properties and results regarding both the coincidence point set and the common fixed point set of any two digitally continuous maps to the case of several (more than two) digitally continuous mappings. Moreover, we study how rigidity may affect these coincidence and homotopy coincidence point sets. Also, we investigate whether an established result by Staecker in Nielsen classical topology regarding the coincidence set for many maps still remains valid in the digital topological setting.
\end{abstract}

\maketitle

\section{Introduction}

Fixed point theory is a vital area in mathematics, it plays a fundamental role in many fields of mathematics from functional and mathematical analysis,  to pure and applied topology etc. In metric spaces, Banach fixed point theorem \cite{Banach} is the pioneer result in this direction, this theorem guarantees not only the existence but also the uniqueness of a fixed point of a certain self map $f$. Moreover, it provides us with a constructive method of finding such a fixed point. Many researches aimed at improving or generalizing  Banach fixed point theorem have been studied in the literature (see. \cite{AbdullahiAzam17A, AbdullahiPoom18, Azametal09, Nadler, SintuKumam11} and others).

Topologically, the main fixed point theorems are: the Brouwer fixed point theorem \cite{Brouwer11abbildung}, the Lefschetz fixed point theorem \cite{Lefschetz26intersections} and the Nielsen fixed point theory \cite{Nielsen27untersuchungen}. These theorems tell us some information regarding the existence of fixed points of some certain self map $f$, the estimation of the fixed points and/or feed us with some information regarding the fixed points of any map $g$ homotopic to $f$ in a topological space $X$. 

Recall that, a topological space $X$ has the fixed point property (FPP, for short) if every continuous function $f : X \longrightarrow X$ has a fixed point. The following is a similar definition that appeared in digital topological setting:
\begin{definition} \cite{Rosen86continuous}
A digital image $(X, \kappa)$ has the FPP if every $\kappa$-continuous $f : X \longrightarrow X$ has a fixed point.
\end{definition}

However, the FPP turns out to be very worthless. This is because one point image is the only digital image with the FPP as was corroborated below \cite{Boxeretal16digital}:

\begin{theorem}
A digital image $(X, \kappa)$ has the FPP if and only if $\#X = 1.$
\end{theorem} 

On the other hand, with digital topology we have answers to questions of how we can apply topological concepts to a binary digital image in a useful and meaningful manner, and also to what extend we can be able do that \cite{Kongetal92guest}. This was first studied by Azriel Rosenfeld in the early 1970s \cite{Rosen79digital}. Digital topology usually provides the theoretical foundations for image processing operations, which include image thinning, image segmentation, boundary detection, contour filling, computer graphics etc (see. \cite{Bertrand94simple, Han05nonproduct, KongRosen96topological}).

The motivation of this research article was born during the Nielsen conference in Kortrijk. The study of coincidence points for many maps has been extensively studied in the literature by many researchers where they considered different kind of spaces or functions. Some of the papers in this direction include (\cite{Gonetal10coincidence, GonRandall17coincidence, MonisWong17obstruction, Staecker11nielsen} and references therein).

The object of this paper is to generalize and improve some properties and results presented in \cite{Abdullahietal19coincidence} regarding both the coincidence point set and the common fixed point set of any two digitally continuous maps to the case of several (more than two) mappings. Moreover, we study how rigidity may affect these coincidence and homotopy coincidence point sets and other related invariants/notions. We also investigate whether a known result as was established by Staecker in \cite{Staecker11nielsen} in the classical Nielsen coincidence theory still remains valid in the digital topological setting. Another generalization we consider, is extending most of the results in \cite{Abdullahietal19coincidence} from self mappings to non-self mappings. Fortunately, as we have anticipated, we were able to show that the non-self mappings may provide us something (results) different from the results already obtained in \cite{Abdullahietal19coincidence} (related to self mappings), some examples were also polished in this regard.

This manuscript is organized into $4$ sections excluding the introduction and conclusion. In Section 2 we review some basic backgrounds needed for this study. In Section 3, we study the coincidence point spectrum for several maps and give some of its characterizations with some illustrative examples. In Section 4, we introduce and study homotopy coincidence point spectrum for several maps between digital images.  Finally, we introduced a special case of the minimum coincidence number called the self coincidence number, we end the section with some of its characterizations.

\section{Preliminaries}

We declare that much of this part is either quoted or paraphrased from the authors' previous article \cite{Abdullahietal19coincidence}. 

Throughout this manuscript, we denote $\mathbb{N}$ and $\mathbb{Z}$ to be the sets of natural numbers and integers respectively. By $id$ and $c$ we mean the identity map (i.e. $id(x) = x$ for all $x \in X$) and the constant map (i.e. $c(x) = x_0$ for all $x \in X$ with $x_0 \in X$ fixed) respectively. Also, we will denote  the number of elements or the cardinality of a set $X$ simply by the notation $\#X.$

A digital image is mathematically a pair $(X, \kappa),$ where $X \subset \mathbb{Z}^n$ for some $n \in \mathbb{N}$ and $\kappa$ is a symmetric and antireflexive relation on $X$ called the adjacency. Moreover, a digital image $(X, \kappa)$ can be viewed as a graph for which $X$ is the vertex set and $\kappa$ the edge set. We usually consider $X$ to be finite set and the adjacency to represent some sort of ``closeness" of the adjacent points in $\mathbb{Z}^n$. 

\subsection{Adjacencies and Neighbourhoods}

We write $x \leftrightarrow_{\kappa} y$ to indicate that $x$ and $y$ are $\kappa$-adjacent and use the notation $x \Leftrightarrow_{\kappa} y$ to indicate that $x$ and $y$ are $\kappa$-adjacent or  equal. Or simply, we use $x \leftrightarrow y$ and $x \Leftrightarrow y$ whenever the adjacency $\kappa$ is understood or unnecessary to mention.

In this paper, we will use the following type of adjacency. For $t \in \mathbb{N}$ with $1 \leq t \leq n$, any 2 (two) points $p = (p_1, p_2, \ldots , p_n)$ and $q = (q_1, q_2, \ldots , q_n)$ in $\mathbb{Z}^n$ (with $p \not= q$) are said to be $\kappa(t, n)$ or $\kappa$-adjacent if at most $t$ of their coordinates differs by $\pm 1,$ and all others coincide. This is commonly called the $c_t$ adjacency, as defined in \cite{BoxerSta19fixed}. Usually, the number of adjacent points (a point can have) is used to denote the $c_t$-adjacencies. For instance, in $1$ dimensional digital image we have $c_1$-adjacency as $2$-adjacency. In dimension 2, we have  $c_1$-adjacency and  $c_2$-adjacency as $4$-adjacency and $8$-adjacency respectively. While in dimension 3, we have  $c_1$-adjacency, $c_2$-adjacency and  $c_3$-adjacency as $6$-adjacency, $18$-adjacency and $26$-adjacency respectively.

Following the Rosenfeld graphical approach, we use the notions of $\kappa$-adjacency relations on $\mathbb{Z}^n$ and a digital $\kappa$-neighborhood as mostly used in the literature. More precisely, the $\kappa$-adjacency relations defined above. Now, we denote and define a digital $\kappa$-neighborhood of a point $p$ in $\mathbb{Z}^n$ as \cite{Rosen79digital}:
\[N_{\kappa}(p) := \{q \, \mid \, q \leftrightarrow_{\kappa} p\}.\]
Sometimes, the following notation is often use to denote a kind of neighborhood that includes the focal point i.e, the digital $\kappa$-neighborhood of $p$ in $\mathbb{Z}^n$ \cite{KongRosen96topological}.
\[N^{\ast}_{\kappa}(p):= \{q \, \mid \, q \Leftrightarrow_{\kappa} p\}.\]

A ``digital interval" is defined as the set $[a, b]_{\mathbb{Z}} = \{n \in \mathbb{Z} \, | \, a \leq n \leq b\}$ together with the $2$-adjacency relation, where $a, b \in \mathbb{Z}$ such that $a \lneq b$ \cite{Boxer94digitally}. Two subsets $(A, \kappa)$ and $(B, \kappa)$ of $(X, \kappa)$ are said to be ``$\kappa$-adjacent" to each other if $A \cap B = \emptyset$ and there exist points $a \in A$ and $b \in B$ such that $a$ and $b$ are $\kappa$-adjacent to each other. A subset $(A, \kappa)$ of $(X, \kappa)$ is called ``$\kappa$-connected" if it is not a union of two disjoint non-empty sets that are not $\kappa$-adjacent to each other, or otherwise $(A, \kappa)$ is called ``$\kappa$-disconnected" \cite{Herman93oriented}. A $\kappa$-component of $x \in X$ is the maximal $\kappa$-connected subset of $(X, \kappa)$ containing the point $x$.

\begin{definition}
An image $(X, \kappa)$  is called ``totally $\kappa$-disconnected" or simply ``totally disconnected" if the $\kappa$-connected component of any point $x\in X$ is the singleton set $\{x\}$.
\end{definition}

\begin{definition}\cite{BoxerSta19fixed}
An $n$-dimensional digital cycle $C_n = \{x_0, x_1, \ldots, x_{n-1}\}$ is an $n$-point digital image where each point $x_i$ is only adjacent to $x_{i-1}$ and $x_{i+1}$.
\end{definition}

\subsection{Digital Continuity and Homotopy}

\begin{definition} \cite{Rosen86continuous}
Let $(X, \kappa_1)$ and $(Y, \kappa_2)$ be digital images. A function $f : X \longrightarrow Y$ is said to be $(\kappa_1, \kappa_2)$-continuous, if for every $\kappa_1$-connected subset $A$ of $X, f(A)$ is a $\kappa_2$-connected subset of $Y.$
\end{definition}

The function $f$ is called digitally continuous whenever $\kappa_1$ and $\kappa_2$ are understood. If $\kappa_1 = \kappa_2 = \kappa$, we say that a function is $\kappa$-continuous to abbreviate $(\kappa, \kappa)$-continuous.

\begin{theorem} \cite{Boxer99classical}
Let $(X, \kappa_1)$ and $(Y, \kappa_2)$ be digital images. Then a function $f : X \longrightarrow Y$ is called $(\kappa_1, \kappa_2)$-continuous if and only if for every $x, y \in X, f(x) \Leftrightarrow_{\kappa_2} f(y)$ whenever $x \leftrightarrow_{\kappa_1} y$.
\end{theorem}

\begin{definition} \cite{Khalimsky87motion}
A digital $\kappa$-path in a digital image $(X, \kappa)$ is a $(2, \kappa)$-continuous function $\gamma : \lbrack 0, m \rbrack_{\mathbb{Z}} \longrightarrow X$. Further, $\gamma$ is called a digital $\kappa$-loop if $\gamma(0) = \gamma(m),$ and the point $p = \gamma(0)$ is the base point of the loop $\gamma.$ Moreover, $\gamma$ is called a trivial loop if $\gamma$ is a constant function.
\end{definition}

\begin{definition} \cite{Boxer94digitally}
A function $f : X \longrightarrow Y$ between digital images $(X, \kappa_1)$ and $(Y, \kappa_2)$ is called an isomorphism if $f$ is a digitally continuous bijection such that $f^{-1}$ is digitally continuous.
\end{definition}

\begin{definition} \cite{Boxer99classical}
Let $(X, \kappa_1)$ and $(Y, \kappa_2)$ be digital images. Suppose that $f, g : X \longrightarrow Y$ are $(\kappa_1, \kappa_2)$-continuous functions, there is a positive integer $m$ and a function $H : X \times \lbrack 0, m \rbrack_{\mathbb{Z}} \longrightarrow Y$ such that:
\begin{enumerate} [label=\arabic*. ]
		\item For all $x \in X, H(x, 0) = f(x)$ and $H(x, m) = g(x)$;
		\item For all $x \in X,$ the induced function $H_x : \lbrack 0, m \rbrack_{\mathbb{Z}} \longrightarrow Y$ defined by 
		\[H_x(t) = H(x, t), \mbox{ for all } t \in \lbrack 0, m \rbrack_{\mathbb{Z}} \]
		is $(c_1, \kappa_1)$-continuous. That is, $H_x(t)$ is a $\kappa$-path in $Y$;
		\item For all $t \in \lbrack 0, m \rbrack_{\mathbb{Z}},$ the induced function $H_t : X \longrightarrow Y$ defined by
		\[H_t(x) = H(x, t), \mbox{ for all } x \in X\]
		is $(\kappa_1, \kappa_2)$-continuous. 
\end{enumerate}
Then $H$ is a digital homotopy (or $\kappa$-homotopy) between $f$ and $g$. Thus, the functions $f$ and $g$ are said to be digitally homotopic (or $\kappa$-homotopic) and denoted by $f \simeq g.$
\end{definition}
Note that if $m = 1,$ then $f$ and $g$ are said to be $\kappa$-homotopic in one step.

\begin{definition} \cite{Khalimsky87motion}
A continuous function $f : X \longrightarrow Y$ is called digitally nullhomotopic in $Y$ if $f$ is digitally homotopic to a constant function $c$. Moreover, a digital image $(X, \kappa)$ is said to be digitally contractible (or $\kappa$-contractible) if its identity map $id$ is digitally nullhomotopic.
\end{definition}

\begin{definition} \cite{BoxerSta19fixed, Haarmannetal15homotopy}
A function $f : X \longrightarrow Y$ is called rigid if no continuous map is homotopic to $f$ except $f$ itself. Moreover, when the identity map $id : X \longrightarrow X$ is rigid, we say that $X$ is rigid.
\end{definition}

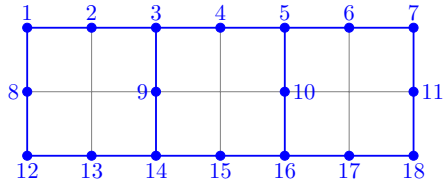
\begin{figure}[h]
\centering
\resizebox{6cm}{2.5cm}{%
		\begin{tikzpicture}
		\draw[step=1cm,gray,very thin] (0,0) grid (6,2);
		\draw [blue, thick] (0,0) -- (0,2);
		\draw [blue, thick] (0,0) -- (6,0);
		\draw [blue, thick] (0,2) -- (6,2);
		\draw [blue, thick] (4,0) -- (4,2);
		\draw [blue, thick] (2,0) -- (2,2);
		\draw [blue, thick] (6,0) -- (6,2);
		\filldraw[blue] (0,0) circle (2pt) node[below]{12};
		\filldraw[blue] (0,1) circle (2pt) node[left]{8};
		\filldraw[blue] (0,2) circle (2pt) node[above]{1};
		\filldraw[blue] (1,0) circle (2pt) node[below]{13};
		\filldraw[blue] (2,0) circle (2pt) node[below]{14};
		\filldraw[blue] (3,0) circle (2pt) node[below]{15};
		\filldraw[blue] (4,0) circle (2pt) node[below]{16};
		\filldraw[blue] (5,0) circle (2pt) node[below]{17};
		\filldraw[blue] (6,0) circle (2pt) node[below]{18};
		\filldraw[blue] (1,2) circle (2pt) node[above]{2};
		\filldraw[blue] (2,2) circle (2pt) node[above]{3};
		\filldraw[blue] (3,2) circle (2pt) node[above]{4};
		\filldraw[blue] (4,2) circle (2pt) node[above]{5};
		\filldraw[blue] (5,2) circle (2pt) node[above]{6};
		\filldraw[blue] (6,2) circle (2pt) node[above]{7};
		\filldraw[blue] (2,1) circle (2pt) node[left]{9};
		\filldraw[blue] (4,1) circle (2pt) node[right]{10};
		\filldraw[blue] (6,1) circle (2pt) node[right]{11};
		\end{tikzpicture}}
	\caption{A 2-dimensional digital image with a 4-adjacency relation.}
	\label{fig1}
\end{figure}

\begin{example} \cite{BoxerSta19fixed}
Let $X = (\lbrack 0, 6 \rbrack_{\mathbb{Z}} \times \{0,2\}) \cup \{(0,1),(2,1),(4,1),(6,1)\}.$ Then $X$ with a $4$-adjacency is a digital image (see Fig. \ref{fig1}). Moreover, $X$ is rigid.
\end{example}

\section{Coincidence Point Spectrum}

In the sequel, we will consider the functions $f_1, \ldots, f_i : X \longrightarrow Y$ to be continuous maps between connected digital images $X$ and $Y$, unless stated otherwise. Now, let's consider the set $C(f_1,\ldots, f_i),$ called the ``coincidence point set" of the maps $f_1, \ldots, f_i.$
\[C(f_1, \ldots, f_i) := \{x \in X \, | \, f_1(x) = \cdots = f_i(x)\}.\]

Whenever we change the maps $f_1, \ldots, f_i$ by a homotopy, we expect the size and shape of the set $C(f_1, \ldots, f_i)$ to vary greatly. So following \cite{Abdullahietal19coincidence}, in this paper, we are also interested in the size of the set $C(f_1, \ldots, f_i).$ One possible tool used in measuring the set $C(f_1, \ldots, f_i)$ is the ``minimum number of coincidence points" (i.e. $MC(f_1, \ldots, f_i)$) which we define as: 
\[MC(f_1, \ldots, f_i) := \min \{\# C(g_1, \ldots, g_i) \, | \, g_i \simeq f_i\}.\]

In the following, we present a generalization of Theorem 3.1 in \cite{Abdullahietal19coincidence} to the case of several maps. This theorem asserts that isomorphism always preserved the number of coincidence points.

\begin{theorem}
Let $X,Y$ be isomorphic digital images. Suppose that $f_1, \ldots, f_i : X \longrightarrow X$ are continuous mappings, for $i \geq 3.$ Then there are continuous mappings $g_1, \ldots, g_i : Y \longrightarrow Y$ such that $\# C(f_1, \ldots, f_i) \, \textup{=} \, \# C(g_1, \ldots, g_i).$
\end{theorem}

\begin{proof}
By hypothesis $\Phi : X \longrightarrow Y$ is an isomorphism. Let $A = C(f_1, \ldots, f_i).$ Then, $\#\Phi(A) = \#A$ since $\Phi$ is one-to-one.
Let's define $g_1, \ldots, g_i : Y \longrightarrow Y$ as $g_i = \Phi \circ f_i \circ \Phi^{-1},$ for $i \geq 3.$ Now, for an arbitrary $y_0 \in \Phi(A),$ let $x_0 = \Phi^{-1}(y_0).$ Then
\begin{equation}
\begin{split}
g_1(y_0) &= \Phi \circ f_1 \circ \Phi^{-1}(y_0)\\
	&= \Phi \circ f_1(x_0)\\
	& \qquad \vdots\\
	&= \Phi \circ f_i(x_0)\\
	&= \Phi \circ f_i \circ \Phi^{-1}(y_0)\\
	&= g_i(y_0).
\end{split}
\end{equation}
Let $B = C(g_1, \ldots, g_i)$, then we obtain
\[\Phi(A) \subseteq B,\]
thus \[\#A \leq \#B.\]
Similarly, for $y_1 \in B$ taken arbitrary and $x_1 = \Phi^{-1}(y_1).$ We have
\begin{equation}
\begin{split}
f_1(y_1) &= \Phi^{-1} \circ g_1 \circ \Phi(y_1)\\
	&= \Phi^{-1} \circ g_1(x_1)\\
	& \qquad \vdots\\
	&= \Phi^{-1} \circ g_i(x_1)\\
	&= \Phi^{-1} \circ g_i \circ \Phi(y_1)\\
	&= f_i(y_1).
	\end{split}
\end{equation}
It implies that \[\Phi^{-1}(B) \subseteq A,\]
which further implies \[\#B \leq \#A.\]
Hence \[\# C(f_1, \ldots, f_i) \, \textup{=} \, \# C(g_1, \ldots, g_i).\]
\qed
\renewcommand{\qed}{} 
\end{proof}

In what follows, we will recall a classical topological concept. Notably, the following result which appeared in \cite{Staecker11nielsen} and asserts that the minimal number of coincidence points between compact manifolds of
the same dimension is always zero.

\begin{theorem}\cite{Staecker11nielsen} \label{th6}
If $X$ and $Y$ are compact manifolds of the same dimension, and $f_1, \ldots, f_i : X \longrightarrow Y$ are maps with $i > 2,$ then these maps can be changed by homotopy so that their coincidence set is empty.
\end{theorem}

So, we believe that it would be interesting to investigate whether or not Theorem \ref{th6} holds in the setting of digital spaces.

\begin{proposition}
Let $X$ and $Y$ be connected images with $\#Y > 1$. Then it is always possible to obtain $C(f_1, \ldots, f_i) = \emptyset$, for some continuous maps $f_1, \ldots, f_i : X \longrightarrow Y$ with $i \geq 2.$
\end{proposition}

\begin{proof}
Let $y_0, y_1 \in Y$ with $y_0 \not= y_1$ since $\#Y > 1.$ Define $f_1, f_2 : X \longrightarrow Y$ as $f_1(x)=y_0$ and $f_2(x)=y_1$ for all $x \in X$ respectively. Then we can see that \[C(f_1, f_2) = C(f_1, f_1, f_2) = C(f_1, \ldots, f_1, f_2) = \emptyset.\]
This implies that whenever we have more than one point in the co-domain, we can define several constant maps to make them coincidence point free.
\qed
\renewcommand{\qed}{} 
\end{proof}

\begin{proposition}\label{prop4}
Let $X$ and $Y$ be connected images. If $\#Y = 1$ then  for continuous maps $f_1, \ldots, f_i : X \longrightarrow Y$ with $i \geq 2,$ $C(f_1, \ldots, f_i) = X.$
\end{proposition}

\begin{proof}
Without loss of generality, we may assume that $Y=\{y_0\},$ then all continuous functions $f_1, \ldots, f_i : X \longrightarrow Y$ has to be the constant function $c : X \longrightarrow Y$ defined as $c(x)=y_0$ for all $x \in X$ and hence meet at $y_0$. this implies that $C(f_1, \ldots, f_i) = C(c, \ldots, c) = X.$
\qed
\renewcommand{\qed}{} 
\end{proof}

\begin{figure}[h]
	\centering
	\subfloat[font=scriptsize][\scriptsize{$(X,4)$.}]{\resizebox{2.75cm}{1.8cm}{%
			\begin{tikzpicture}
			\draw [blue, thick] (0,0) -- (0,1);
			\draw [blue, thick] (0,0) -- (1,0);
			\draw [blue, thick] (0,1) -- (1,1);
			\draw [blue, thick] (1,0) -- (1,1);
			\filldraw[blue] (0,0) circle (2pt) node[left]{$x_0$};
			\filldraw[blue] (0,1) circle (2pt) node[left]{$x_1$};
			\filldraw[blue] (1,1) circle (2pt) node[right]{$x_2$};
			\filldraw[blue] (1,0) circle (2pt) node[right]{$x_3$};
			\end{tikzpicture}}
	}%
	\qquad
	\subfloat[][\scriptsize{$(Y,4)$.}]{
		\begin{tikzpicture}
		\draw [blue, thick] (1,0) -- (1,1);
		\draw [blue, thick] (0,0) -- (1,0);
		\draw [blue, thick] (1,0) -- (2,0);
		\filldraw[blue] (0,0) circle (2pt) node[left]{$y_0$};
		\filldraw[blue] (1,0) circle (2pt) node[below]{$y_1$};
		\filldraw[blue] (1,1) circle (2pt) node[above]{$y_3$};
		\filldraw[blue] (2,0) circle (2pt) node[right]{$y_2$};
		\end{tikzpicture}
	}\caption{The two adjacency relations on any 2-dimensional digital image.}%
	\label{figex}%
\end{figure}
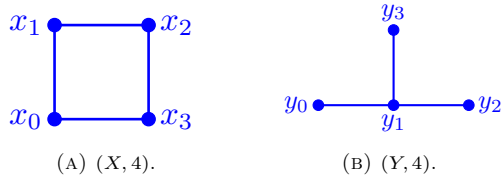

\begin{example}
Let $X = \{x_0,x_1,x_2,x_3\}$ and $Y = \{y_0,y_1,y_2,y_3\}$ be digital images with $4$-adjacency in both (see Fig. \ref{figex}. (A) and (B) respectively). Let $f : X \longrightarrow Y$ be defined as: $f(x_0)=f(x_2)=y_1, f(x_1)=y_0, f(x_3)=y_2$, $g : X \longrightarrow Y$ be defined as $g(x_0)=y_0, g(x_1)=g(x_3)=y_1, g(x_2)=y_3$ and $c : X \longrightarrow Y$ be defined as $c(x)=y_3$ for all $x \in X$. Then the mappings $f,g,c$ are all continuous. Moreover, $C(f, g, c) = \emptyset$. 
\end{example}

Let $X$ and $Y$ be connected images and $f_1, \ldots, f_i : X \longrightarrow Y$ are continuous mappings with $i \geq 3.$ One question that can naturally arise is that, can these mappings be deformed by a homotopy so that $C(f_1, \ldots, f_i) = \emptyset$? This is not always possible in the context of digital topology as shown by the following example. 

\begin{example}
Let $X$ be a rigid digital image and $id$ be the identity map on $X$. Then $C(id, id, id) \not= \emptyset$ and since $id$ is not homotopic to any other map, then $C(id, id, id)$ can never be made coincidence free when changed by a homotopy.
\end{example}

Lets consider the ``coincidence point spectrum" of $i$ mappings $f_1, \ldots, f_i : X \longrightarrow Y,$ which we denote and define as:
\[CS_i(X,Y) = \{\# C(f_1, \ldots, f_i) \, | \, f_1, \ldots, f_i : X \longrightarrow Y \mbox{ are continuous}\}.\]
More broadly, we may consider the following set, which we call the ``coincidence point spectrum".
\[CS(X,Y) = \cup_i \, CS_i(X,Y).\]

As a special case whenever $X=Y$, we simply have the coincidence point spectrum to be the following set:
\[CS_i(X) = \{\# C(f_1, \ldots, f_i) \, | \, f_1, \ldots, f_i : X \longrightarrow X \mbox{ are continuous}\},\]
and more broadly
\[CS(X) = \cup_i \, CS_i(X).\]

Following a similar pattern to the assertions in \cite{Abdullahietal19coincidence}, we will consider the extended notion of ``common fixed point set" for multiple maps as:
\begin{equation}\label{eq1}
CF(f_1, \ldots, f_i) := \{x \in X \, | \, f_1(x) = \cdots = f_i(x) = x\}.
\end{equation}

We may also define the ``minimum number of common fixed points of $i$ mappings $f_1, \ldots, f_i$ as:
\[MCF(f_1, \ldots, f_i) := \min \{\# CF(g_1, \ldots, g_i) \, | \, g_j \simeq f_j\}.\]

Another important invariant is the ``common fixed point spectrum" of the self mappings $f_1, \ldots, f_i : X \longrightarrow X,$ which we denote and define as:
\[CFS_i(X) = \{\# CF(f_1, \ldots, f_i) \, | \, f_1, \ldots, f_i : X \longrightarrow Y \mbox{ are continuous}\}.\]
More broadly, we may consider the following set, which we call the ``common fixed point spectrum" for all possible $i$.
\[CFS(X) = \cup_i \, CFS_i(X).\]

Note that, we can view the generalized concept of ``common fixed point set" for several maps simply as:
\begin{equation}\label{eq2}
C(f_1, \ldots, f_{i-1}, id) := \{x \in X \, | \, f_1(x) = \cdots = f_{i-1}(x) = x\}.
\end{equation}

Observe that (\ref{eq1}) and (\ref{eq2}) are identical. Note also that, if $i = 3$, the generalized notion will coincide with the previous concept of common fixed point of two maps $f_1$ and $f_2$ studied in \cite{Abdullahietal19coincidence}.

\begin{remark} 
If $f_1 = \cdots = f_{i-1} = f$ and $f_i = id$ with $i \geq 2$, then \[C(f, f, \ldots, f, id) = \textup{Fix($f$)}.\]
\end{remark}

\begin{theorem} 
Let $X$, $Y$ be connected images and $f_1, \ldots, f_{i+1} : X \longrightarrow Y$ be continuous maps with $i \geq 2.$ Then \[C(f_1, \ldots, f_{i+1}) \subseteq C(f_1, \ldots, f_i).\]
\end{theorem}

\begin{proof}
Let $x \in C(f_1, \ldots, f_{i+1})$, it is enough to show that $x \in C(f_1, \ldots, f_i).$ Since $x \in C(f_1, \ldots, f_{i+1})$ it implies that $f_1(x) = \cdots = f_i(x) = f_{i+1}(x),$ which further implies that $x \in C(f_1, \ldots, f_i).$
\qed
\renewcommand{\qed}{} 
\end{proof}

\begin{lemma} \label{lemma1}
Let $X$, $Y$ be digital images. Then \[\#X \in CS_i(X,Y).\] If additionally $\#Y>1$, then \[0 \in CS_i(X,Y).\]
\end{lemma}

\begin{proof}
Let $x \in X$ then $x \in C(f,\ldots,f)$ if and only if  $f(x)=\cdots=f(x)$. Since $x$ is arbitrary taken, we then obtain $C(f,\ldots,f) = X$ which obviously implies that $\#X \in CS_i(X,Y)$ for all $i.$ 
	
For $0 \in CS_i(X,Y),$ since $\#Y>1$ we have at least two points  $y_0,y_1\in Y$ and can define the constant maps $c(x)=y_0$ and $c^{\prime}(x)=y_1$ for all $x \in X,$ such that $C(c,\ldots,c^{\prime},c^{\prime})=\emptyset$ which completes the proof.
\qed
\renewcommand{\qed}{} 
\end{proof}

\begin{lemma} \label{lemma2}
Let $X$ and $Y$ be digital images with $n=\#X$ and $Y$ having at least one pair of adjacent points. Then \[CS_2(X,Y) = \{0,1,\ldots,n\}.\]
\end{lemma}

\begin{proof}
From Lemma \ref{lemma1} we have $0 \in CS_2(X,Y).$ Now, let $a$ and $b$ be any two adjacent points in $Y$. Let $m$ be any number $1 \leq m \leq  n=\#X$. Now, define a map $f: X \longrightarrow Y$ as $f(x)=a$ for $x\in\{x_1,\ldots, x_m\}$ and $f(x)=b$ for $x\in\{x_{m+1},\ldots, x_n\}$. Then $f$ is automatically continuous and $m=\#C(f,c),$ where $c$ is the constant function with value $a.$ (i.e $c(x)=a$ for all $x$). Thus $m \in CS_2(X,Y).$ As a special case, even when $n=1$ we still have $CS_2(X,Y) = \{0,1\},$ which is obvious from Lemma \ref{lemma1}.  Moreover, when $X=Y$ with $\#X > 1,$ we have $CS_2(X) = \{0,1,\ldots,n\}.$
\qed
\renewcommand{\qed}{} 
\end{proof}

\begin{theorem} \label{tincc}
Let $X$ and $Y$ be digital images. Then \[CS_i(X,Y) \subseteq CS_{i+1}(X,Y).\]
If additionally $Y$ contains at least one pair of adjacent points, then \[CS_i(X,Y) = CS_{i+1}(X,Y).\]
\end{theorem}

\begin{proof}
Let $m \in CS_i(X,Y)$, it is enough to show that $m \in CS_{i+1}(X,Y).$ Since $m \in CS_i(X,Y)$ it implies that there exist some continuous functions $f_1, \ldots, f_i : X \longrightarrow Y$ such that $m = \# C(f_1, \ldots, f_{i})$. This further implies that $m = \# C(f_1, \ldots, f_i, f_i)$ which immediately proves that $m \in CS_{i+1}(X,Y)$ as required. The equality follows from Lemma \ref{lemma2}.
\qed
\renewcommand{\qed}{} 
\end{proof}

\begin{theorem} \label{tincf}
Let $X$ be a digital image. Then \[F(X) \subseteq CS_2(X).\]
\end{theorem}

\begin{proof}
Let $m \in F(X)$, it is enough to show that $m \in CS_2(X).$ Since $m \in F(X)$ it implies that there exist a continuous map $f : X \longrightarrow X$ such that $m = \# Fix(f)$. This further implies that $m = \# C(f, id)$ which immediately proves that $m \in CS_2(X)$ as required.
\qed
\renewcommand{\qed}{} 
\end{proof}

\begin{figure}[h]%
\centering
\subfloat[font=scriptsize][\scriptsize{$(X,6)$.}]{\resizebox{3.5cm}{3cm}{%
	\begin{tikzpicture}
			\draw[step=1cm,white,very thin] (0,0) grid (3,3);
			\draw [blue, thick] (0,0) -- (0,2);
			\draw [blue, thick] (0,2) -- (1,3);
			\draw [blue, thick] (1,3) -- (3,3);
			\draw [blue, thick] (1,1) -- (1,3);
			\draw [blue, thick] (0,0) -- (2,0);
			\draw [blue, thick] (0,2) -- (2,2);
			\draw [blue, thick] (3,1) -- (3,3);
			\draw [blue, thick] (2,0) -- (2,2);
			\draw [blue, thick] (2,0) -- (3,1);
			\draw [blue, thick] (1,1) -- (3,1);
			\draw [blue, thick] (2,2) -- (3,3);
			\draw [blue, thick] (0,0) -- (1,1);
			\filldraw[blue] (0,0) circle (3pt) node[left]{$x_0$};
			\filldraw[blue] (1,1) circle (3pt) node[left]{$x_3$};
			\filldraw[blue] (0,2) circle (3pt) node[left]{$x_1$};
			\filldraw[blue] (2,2) circle (3pt) node[right]{$x_5$};
			\filldraw[blue] (2,0) circle (3pt) node[below]{$x_4$};
			\filldraw[blue] (1,3) circle (3pt) node[above]{$x_2$};
			\filldraw[blue] (3,3) circle (3pt) node[right]{$x_6$};
			\filldraw[blue] (3,1) circle (3pt) node[right]{$x_7$};
			\end{tikzpicture}}
	}%
	\qquad
	\subfloat[][\scriptsize{$(Y,6)$.}]{
		\resizebox{3.5cm}{3cm}{%
			\begin{tikzpicture}
			\draw[step=1cm,white,very thin] (0,0) grid (3,3);
			\draw [blue, thick] (0,2) -- (1,3);
			\draw [blue, thick] (1,3) -- (3,3);
			\draw [blue, thick] (1,1) -- (1,3);
			\draw [blue, thick] (0,2) -- (2,2);
			\draw [blue, thick] (3,1) -- (3,3);
			\draw [blue, thick] (2,0) -- (2,2);
			\draw [blue, thick] (2,0) -- (3,1);
			
			\draw [blue, thick] (1,1) -- (3,1);
			\draw [blue, thick] (2,2) -- (3,3);
			\filldraw[blue] (1,1) circle (3pt) node[left]{$y_0$};
			\filldraw[blue] (0,2) circle (3pt) node[left]{$y_6$};
			\filldraw[blue] (2,2) circle (3pt) node[above]{$y_4$};
			\filldraw[blue] (2,0) circle (3pt) node[below]{$y_1$};
			\filldraw[blue] (1,3) circle (3pt) node[above]{$y_5$};
			\filldraw[blue] (3,3) circle (3pt) node[above]{$y_3$};
			\filldraw[blue] (3,1) circle (3pt) node[right]{$y_2$};
			\end{tikzpicture}}
	}\caption{Contractible 3-dimensional digital images with 6-adjacency relation.}%
	\label{fig3}%
\end{figure}
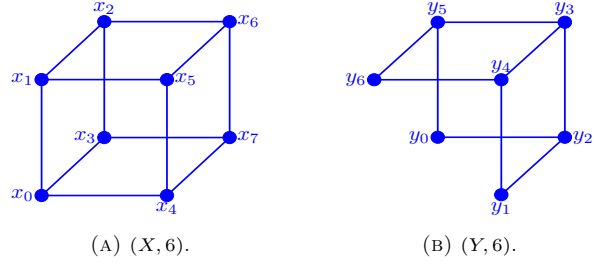

Now, it is natural to wonder if there is any interesting relationship between $F(X)$ and $CS_2(X,Y)$.

\begin{example}\label{exm1}
Let $X$ and $Y$ be the digital images in Fig. \ref{fig3}. Then, $Y \subset X$ and from \cite{BoxerSta19fixed}, we have \[F(X)=\{0,1,2,3,4,5,6,8\}.\]
By Lemma \ref{lemma2}, we have \[CS_2(X,Y)=CS_2(X)=\{0,1,2,3,4,5,6,7,8\}.\]
Moreover, if we let $Z=\{z_0\}$ then from Proposition \ref{prop4}, we have \[CS_2(X,Z)=\{8\}.\] Therefore 
\[CS_2(X,Y) \not= F(X) \not= CS_2(X,Z).\]
\end{example}

In Example \ref{exm1} we have shown that $F(X)$ is different from both $CS(X)$ and $CS(X,Y)$. Now, in the following example we demonstrate further that $CS(X,Y)$ and $CS(X)$ are not generally the same concepts as suggested by the preceding example. 

Moreover, we can also note that when $Y$ is a single point image, we always have \[CS(X,Y)=\{\#X\}\] and \[1 = \#CS(X,Y) \leq \#CS(X).\]

\begin{example}
Let $X$ be any connected digital images and $Y=\{y_0\}$ (i.e. $Y$ is a single point image). Then, we always have \[CS_i(X,Y)= \{\# X\} = CS(X,Y), \hspace{0.5cm} \mbox{for any } i\geq 2.\]
\end{example}

From the above example, we observe that the following inclusion 
\[F(X) \subseteq CS_2(X,Y).\]
does not always holds whenever $Y\not=X.$

Another instance where $F(X)$ can be different from $CS(X)$ is the following example.

\begin{example}
Let $C_n$ be any digital cycle with $n$ points. From \cite{BoxerSta19fixed}, we have continuous functions say $f$ having $n$ number of fixed points, where $n=\{1,2,3,4\}$ so that $C(id, f)$ will give us the following. 
\[CS(C_n) = F(C_n) = \left\{ \begin{array}{ll}
	\{1\}, & \mbox{if $ n = 1$};\\
	\{0,\ldots, n\}, & \mbox{if $ n \in \{2,3,4\}$.} \end{array} \right. \]
Moreover, from \cite{BoxerSta19fixed} and Lemma \ref{lemma2}, we have \[\{0,1,\ldots,\floor{\frac{n}{2}} +1,n\} \subset \{0,1,\ldots, n\} = CS(C_n), \hspace{0.5cm} \mbox{for any } n\geq 5,\] where $F(C_n)=\{0,1,\ldots,\floor{\frac{n}{2}} +1,n\}$ and $\floor{\cdot}$ means the floor function.
\end{example}

\begin{theorem} \label{tdis}
Let $X$ be a connected digital image and $Y$ be totally disconnected digital image with $\#Y>1.$ Then \[CS_i(X,Y) = \{0, \#X\}, \hspace{0.5cm} \mbox{for all } i \geq 2.\]
\end{theorem}

\begin{proof}
For $i=2$, $CS_2(X,Y) = \{0, \#X\}$ follows immediately from Lemma \ref{lemma2}. Since $Y$ has no pair of adjacent points, the only contributors to the coincidence point spectrum are the coincidence of distinct constants maps and self-coincidence of any map $f$. Subsequently, we have $CS_i(X,Y) = \{0, \#X\}$ for all $i$ as required.
\qed
\renewcommand{\qed}{} 
\end{proof}

If $X$ in Theorem \ref{tdis} is also totally disconnected then we can get more elements in the coincidence point spectrum. For example, take $X = Y$
to be a set of two non-adjacent points, then \[CS_2(X,Y) = \{0,1,2\}.\] The one coincidence point is obtained by taking the coincidence of the identity map $id$ and a constant map $c$.

Now, the only case where we are not sure what exactly the set $CS_2(X,Y)$ will be is when $Y$ is totally disconnected i.e. $Y$ has no pair of adjacent points with $\#Y > 1$ and $X$ is just disconnected.  However, we think probably that in this situation the set $CS_2(X,Y)$ may be obtained in terms of the number of points in each component of $X.$

From Theorem \ref{tincc}. Example \ref{exm2} and Theorem \ref{tdis}, we have the following.

\begin{corollary} 
	Let $X$ and $Y$ be digital images. Then \[CS_2(X,Y) = CS_3(X,Y) = \cdots = CS_i(X,Y), \hspace{0.5cm} \mbox{for all } i \geq 2,\]
	in any of the following cases:
	\begin{itemize}
		\item $X$ is connected;
		\item $\#Y = 1$;
		\item $Y$ has at least one pair of adjacent points.
	\end{itemize}
\end{corollary}

\begin{remark}
	We notice that the only case in which probably $CS_i(X,Y)$ is not equal to $CS_{i+1}(X,Y)$ for any $i$ is when $X$ is disconnected and $Y$ is totally disconnected with $\#Y > 1.$ 
\end{remark}

To this end, we still believed that even in that case the coincidence point spectra are all equal for any number of maps. However, we are still yet to prove it as it seems difficult to do. Therefore, we conjectured the following.

\begin{conjecture}
	Let $X$ and $Y$ be digital images and in addition, let $X$ be disconnected and $Y$ be totally disconnected with $\#Y > 1.$ Then \[CS_2(X,Y) = CS_3(X,Y) = \cdots = CS_i(X,Y), \hspace{0.5cm} \mbox{for all } i \geq 2.\]
\end{conjecture}

\section{Homotopy Coincidence Point Spectrum}

In this section, we study homotopy coincidence point spectrum, which is an important invariant that will give us an estimate on all possible number of coincidence points of the maps $f_1, \ldots, f_i$ when they are all allowed to be changed by a homotopy. 

So, for some maps $f_1, \ldots, f_i : X \longrightarrow Y,$ we may define the set $HCS(f_1, \ldots, f_i ),$ which we call the ``homotopy coincidence point spectrum" of the $i$ functions $f_1, \ldots, f_i$ as:
\[HCS(f_1,\ldots, f_i) = \{\# C(g_1, \ldots, g_i)  \, | \, g_j \simeq f_j\} \subseteq \{0, 1,  \ldots ,\#X\}.\]

\begin{remark}
\begin{enumerate}[label=(\roman*) ] 
	\item $MC(f_1, \ldots, f_i) = \min{HCS(f_1, \ldots, f_i)};$
	\item Moreover, both $MC(f_1, \ldots, f_i)$ and $HCS(f_1, \ldots, f_i)$ are homotopy invariants for any continuous maps $f_1\ldots, f_i, (i \geq 2)$.
\end{enumerate}
\end{remark}

Now, for some mappings $f_1, \ldots, f_i : X \longrightarrow X$ we may consider the following set, which we call the ``homotopy common fixed point spectrum" $i$ mappings $f_1, \ldots, f_i$ as:
\[HFS(f_1, \ldots, f_i) = \{\# CF(g_1, \ldots, g_i) \, | \, g_j \simeq f_j\}.\]

\begin{theorem} \label{tinch}
Let $X$ and $Y$ be digital images and $f_1, \ldots, f_i : X \longrightarrow Y$ be continuous mappings with $i \geq 2.$ Then \[HCS(f_1, \ldots, f_i) \subseteq HCS(f_1, \ldots, f_i, f_i).\]
\end{theorem}

\begin{proof}
Let $m \in HCS_{i}(f_1, \ldots, f_i)$, we show that $m \in HCS_{i+1}(f_1, \ldots, f_i).$ Since $m \in HCS_{i}(f_1, \ldots, f_i)$ it implies that there exist some functions $g_1, \ldots, g_i : X \longrightarrow Y$ such that $g_j$ is homotopic to $f_j$ and $m = \# C(g_1, \ldots, g_{i})$. This further implies that $m = \# C(g_1, \ldots, g_i, g_i)$ which immediately proves that $m \in HCS_{i+1}(f_1, \ldots, f_i)$ as required.
\qed
\renewcommand{\qed}{} 
\end{proof}

Theorem \ref{tinch} above can actually be equality and we know of no example where those sets are not equal. Moreover, we could also see that when $Y$ has at least one pair of adjacent points, we have
\[HCS(c,c,\dots,c) = \{0,...,\#X\}.\]
This follows from a similar argument in the proof of Lemma \ref{lemma2}.

\begin{proposition} 
Let $X$ be a rigid digital image. Then \[HCS(id, \ldots, id) = HCS(id, \ldots, id, id) = \{\#X\}.\]
\end{proposition}

\begin{proof}
Since $X$ is rigid, the only homotopic mapping to the identity map $id$ is $id$ itself. This implies that $C(id,\ldots,id)=X$ and obviously gives us \[HCS(id, \ldots, id) = HCS(id, \ldots, id, id) = \{\#X\},\] as required.
\qed
\renewcommand{\qed}{} 
\end{proof}

\section{Self Coincidence Number}

In this section, we introduce some special situation regarding the minimum number of coincidence points of two or more (several) maps and present some of its characterizations. Now, consider the following invariant.
\[m_j(f) = MC(f,f,\dots,f),\]
where the subscript $j$ stands for $j$ different copies of the mapping $f.$

\begin{remark}
As a special case, when $f=id$. We simply consider $m_j(id)$ to be defined and denoted as: \[m_j(X) = MC(id,id,\dots,id),\] where the subscript $j$ stands for $j$ different copies of the identity map. For instance, when $j=2$ and $j=3$ we have $m_2(X) = MC(id,id)$ and $m_3(X) = MC(id,id,id)$ respectively.
\end{remark}

Note that, if $X$ is rigid image then $m_j(X) = \# X$ for all $j,$ and when $X = C_n,$ (i.e. digital cycle of $n$ points) we have $m_j(X) = 0$ for all $j.$ Here in both cases, the sequence of numbers (i.e. $\{m_j(X)\}_{j=1}^{\infty}$) is constant.

\begin{theorem}
The sequence $\{m_j(X)\}_{j=1}^{\infty}$ is non-increasing (i.e. $m_{j+1}(X) \le m_j(X)$ for all $j \in \mathbb{N}$).
\end{theorem}

\begin{proof}
Let $m_j(X) = MC(id,id,\dots,id)$ with $j$ copies of the identity map. Then there exist some mappings $f_1,\ldots, f_j$ which are homotopic to $id$ such that $m_j(X) = \# C(f_1,\dots,f_j)$. Now, consider $m_{j+1}(X) = MC(id,id,\dots,id)$ with $j+1$ copies of the identity map. Then by Theorem \ref{tincc} and \ref{tinch} we know that the coincidence spectrum is growing larger every time and hence the coincidence set. This further implies that its minimum must be non-increasing and this completes the proof. 
\qed
\renewcommand{\qed}{} 
\end{proof}

\section{Conclusion}

In this manuscript, we generalized results regarding both the coincidence point set and the common fixed point set of any two digitally continuous maps which appeared in \cite{Abdullahietal19coincidence} to the case of multiple maps (with more than two digitally continuous mappings). Particularly, we studied the coincidence point spectrum for several maps and presented some of its characterizations, some illustrative examples were also provided to support our results. Later, we introduced and study the homotopy coincidence point spectrum for several maps between digital images. Also, we investigated whether a known result in Nielsen classical topology regarding the coincidence set for many maps as established by Staecker in \cite{Staecker11nielsen} still remains valid in the digital topological setting. However, we have shown that this theorem fails to generally be true in digital topology. In particular, we have shown that whenever an image is rigid the coincidence point set of the identity maps can never be made coincidence free by a homotopy.  Lastly, we introduced a special case of the minimum coincidence number which we called the self coincidence number, we then studied the effect of rigidity on this number.

\section{Acknowledgement}

The authors acknowledge the financial support provided by the Center of Excellence in Theoretical and Computational Science (TaCS-CoE), KMUTT. The first author was supported by the ``Petchra Pra Jom Klao Ph.D. Research Scholarship from King Mongkut's University of Technology Thonburi" (Grant No.: 35/2017). Finally, we would like to thank Dr. Peter Christoper Staecker (Fairfield University) for valuable comments, suggestions and careful reading of this manuscript.

%
%
%
%

\bibliographystyle{plain} 
\bibliography{reference}         

\vspace{0.5cm}

\end{document}